\documentclass[11pt]{article}
\usepackage{amsfonts}

\newtheorem{theorem}{Theorem}

\newtheorem{lemma}[theorem]{Lemma}

\newenvironment{proof}[1][Proof]{\textbf{#1.} }{\
\rule{0.5em}{0.5em}} \textheight 24cm \textwidth 16cm
\title{Rational Extensions of $C(X)$ via Hausdorff Continuous Functions}
\author{Roumen Anguelov\\
Department of Mathematics and Applied Mathematics\\
University of Pretoria, South Africa\\
roumen.anguelov@up.ac.za}
\date{}

\begin{document}

\maketitle

\begin{abstract}
The ring operations and the metric on $C(X)$ are extended to the
set $\mathbb{H}_{nf}(X)$ of all nearly finite Hausdorff continuous
interval valued functions and it is shown that
$\mathbb{H}_{nf}(X)$ is both rationally and topologically
complete. Hence, the rings of quotients of $C(X)$ as well as their
metric completions are represented as rings of Hausdorff
continuous functions.
\end{abstract}



\section{Introduction}
Let $B$ be a commutative ring with identity and let $A$ be a
subring of $B$ having the same identity. The ring $B$ is called a
\textit{rational extension} or a \textit{ring of quotients} of $A$
if for every $b\in B$ the subring $b^{-1}A=\{a\in A:ba\in A\}$ is
rationally dense in $B$, that is, $b^{-1}A$ does not have nonzero
annihilators. Any ring $A$ has a maximal rational extension
$\mathcal{Q}(A)$. The ring $\mathcal{Q}(A)$ is also called a
complete or total ring of quotients of $A$. The classical ring of
quotients
\[\mathcal{Q}_{cl}(A)=\left\{\frac{p}{q}:p,q\in A,\ q
\mbox{ is not a zero divisor}\right\}
\]
is, in general, a subring of $\mathcal{Q}(A)$. A ring without
proper rational extension is called \textit{rationally complete}.

It is well known that the ring $C(X)$ of all continuous real
functions on a topological space $X$ is not rationally complete.
Our goal is to represent the rational extensions of $C(X)$ as
rings of functions defined on the same domain, namely as rings of
Hausdorff continuous (H-continuous) functions on $X$. The
H-continuous functions are a special class of extended interval
valued functions, that is, their range, or target set, is
$\mathbb{I\overline{R}}=\{[\underline{a},\overline{a}]:\underline{a},
\overline{a}\in\overline{\mathbb{R}}=\mathbb{R}\cup\{\pm\infty\},\
\underline{a}\leq\overline{a}\}$. Due to a certain minimality
condition, they are not unlike the usual continuous functions. For
instance, they are completely determined by their values on any
dense subset of the domain. More precisely, for H-continuous
functions $f$, $g$ and a dense subset $D$ of $X$ we have
\begin{equation}\label{ident}
f(x)=g(x),\ x\in D\ \Longrightarrow f(x)=g(x),\ x\in X.
\end{equation}
Within the realm of Set-Valued Analysis, the H-continuous
functions can be identified with the minimal upper semi-continuous
compact set-valued (usco) maps from $X$ into
$\overline{\mathbb{R}}$, \cite{musco}.

We extend the ring structure on $C(X)$ to the set
$\mathbb{H}_{nf}(X)$ of all nearly finite H-continuous functions
following the approach in \cite{NMA2007}. We show further that the
ring $\mathbb{H}_{nf}(X)$ is rationally complete. Hence, it
contains all rational extensions of $C(X)$. The maximal and
classical rings of quotients are represented as subrings of
$\mathbb{H}_{nf}(X)$.

Let $||\cdot||$ denote the supremum norm on the set $C_{bd}(X)$ of
the bounded continuous functions. Then
\begin{equation}\label{rho}
\rho(f,g)=\left|\!\left|\frac{f-g}{1+|f-g|}\right|\!\right|
\end{equation}
is a metric on $C(X)$, which can be extended further to
$\mathcal{Q}(C(X))$. We prove that its completion
$\overline{\mathcal{Q}}(C(X))$ with respect to this metric is
exactly $\mathbb{H}_{nf}(X)$. The Dedekind completions $C(X)^\#$
and $C_{bd}(X)^\#$ of $C(X)$ and $C_{bd}(X)$ with respect to the
pointwise partial order, being subrings of $\mathcal{Q}(C(X))$,
also admit a convenient representation as rings of H-continuous
functions. These results significantly improve earlier results of
Fine, Gillman and Lambek \cite{FGL}, where the considered rings
are represented as direct limits of rings of continuous functions
on dense subsets of $X$ and $\beta X$.\vspace{3pt}

\section{The ring of nearly finite Hausdorff continuous functions}

We recall \cite{Sendov} that an interval function
$f:X\to\mathbb{I\overline{R}}$ is called \textit{S-continuous} if
its graph is a closed subset of $X\times\mathbb{\overline{R}}$. An
interval function $f:X\to\mathbb{I\overline{R}}$ is Hausdorff
continuous (H-continuous) if it is an S-continuous function which
is minimal with respect to inclusion, that is, if
$\varphi:\Omega\to\mathbb{I\overline{R}}$ is an S-continuous
function then $\varphi\subseteq f$ implies $\varphi=f$. Here the
inclusion is considered in a point-wise sense. We denote by
$\mathbb{H}(X)$ the set of H-continuous functions on $X$.

Given an interval
$a=[\underline{a},\overline{a}]\in\mathbb{I\overline{R}}$,
 \[
w(a)=\left\{
\begin{tabular}
[c]{lll}%
$\overline{a}-\underline{a}$ & if & $\underline{a},\overline{a}$ finite,\\
$+\infty$ & if & $\underline{a}<\overline{a}=+\infty$ or $\underline{a}%
=-\infty<\overline{a}$,\\
0 & if & $\underline{a}=\overline{a}=\pm\infty,$%
\end{tabular}
\right.
\]
is the width of $a$, while
$|a|=\max\{|\underline{a}|,|\overline{a}|\}$ is the modulus of
$a$. An interval $a$ is called proper interval, if $w(a)>0$ and
point interval, if $w(a)=0$. Identifying $a\in
\mathbb{\overline{R}}$ with the point interval $[a,a]\in
\mathbb{I\overline{R}}$, we consider $\mathbb{\overline{R}}$ as a
subset of $\mathbb{I\overline{R}}$. H-continuous functions are
similar to the usual real valued continuous real functions in that
they assume proper interval values only on a set of First Baire
category, that is, for every $f\in\mathbb{H}(X)$ the set
$W_f=\{x\in X : w(f(x))>0\}$ is countable union of closed and
nowhere dense set, \cite{QM2004}. Furthermore, $f$ is continuous
on $X\setminus W_f$. If $X$ is a Baire space, $X\setminus W_f$ is
also dense in $X$. Thus, in this case, $f$ is completely
determined by its point values. This approach is used for defining
linear space operations \cite{RC} and ring operations
\cite{NMA2007} for H-continuous functions. Here we do not make any
such assumption on $X$. Hence the approach is different.

For every S-continuous function $g$ we denote by $\langle
g\rangle$ the set of H-continuous functions contained in $g$, that
is,
\[ \langle g\rangle=\{f\in\mathbb{H}(\Omega):f\subseteq g\}.
\]
Identifying $\{f\}$ with $f$ we have $\langle f\rangle =f$
whenever $f$ is H-continuous. The S-continuous functions $g$ such
that the set $\langle g\rangle$ is a singleton, that is, it
contains only one function, play an important role in the sequel.
In analogy with the H-continuous functions, which are minimal
S-continuous functions, we call these functions
\textit{quasi-minimal S-continuous functions} \cite{musco}. The
following characterization of the quasi-minimal S-continuous
functions is useful.

\begin{theorem}\label{tqminchar}
Let $f$ be an S-continuous function on $X$. Then $f$ is quasi-
minimal S-continuous function if and only if for every
$\varepsilon>0$ the set
\[
W_{f,\varepsilon}=\{x\in X:w(f(x))\geq\varepsilon\}
\]
is closed and nowhere dense in $X$.
\end{theorem}
\begin{proof}
Let us assume that an S-continuous function $f$ is not
quasi-minimal. Then there exist H-continuous functions
$\phi=[\underline{\phi},\overline{\phi}]$ and
$\psi=[\underline{\psi},\overline{\psi}]$, $\phi\neq \psi$ such
that $\phi\subseteq f$ and $\psi\subseteq f$. Due to the
minimality property of H-continuous functions the set $\{x\in
X:\phi(x)\cap\psi(x)=\emptyset\}$ is open and dense subset of $X$.
Let $a\in X$ be such that $\phi(a)\cap\psi(a)=\emptyset$. Without
loss of generally we may assume that
$\overline{\psi}(a)<\underline{\phi}(a)$. Let
$\varepsilon=\frac{1}{3}(\underline{\phi}(a)-\overline{\psi}(a))$.
Using that $\overline{\phi}$ and $\underline{\phi}$ are
respectively upper semi-continuous and lower semi-continuous
functions, there exists an open neighborhood $V$ of $a$ such that
$\overline{\psi}(x)<\overline{\psi}(a)+\varepsilon<\underline{\phi}(a)-\varepsilon<\underline{\phi}(x)$
, $x\in V$. Then
\[
w(f)(x)\geq
\underline{\phi}(x)-\overline{\psi}(x)>\underline{\psi}(a)-\overline{\phi}(a)-2\varepsilon=\varepsilon,\
\ x\in V.
\]
Hence $V\subset W_{f,\varepsilon}$, which implies that
$W_{f,\varepsilon}$ is not nowhere dense. Therefore, if
$W_{f,\varepsilon}$ is nowhere dense for every $\varepsilon>0$
then $f$ is quasi-minimal.

Now we prove the inverse implication, that is, that for any
S-continuous quasi-minimal function $f$ and $\varepsilon>0$ the
set $W_{f,\varepsilon}$ is closed and nowhere dense. Assume the
opposite. Since for an S-continuous function $f$ the set
$W_{f,\varepsilon}$ is always closed, this means that there exists
an S-continuous function $f=[\underline{f},\overline{f}]$ and
$\varepsilon>0$ such that $W_{f,\varepsilon}$ is not nowhere
dense. Hence there exists an open set $V$ such that $V\subseteq
W_{f,\varepsilon}$. Then there exist an H-continuous functions
$\phi$ on $V$ such that $\phi(x)\subseteq
[\underline{f}(x),\overline{f}(x)-\varepsilon]$, $x\in V$. The we
have $\phi(x)+\varepsilon \subseteq
[\underline{f}(x),\overline{f}(x)-\varepsilon]$, $x\in V$. It is
easy to see that the functions $\phi$ and $\phi+\varepsilon$ can
both be extended from $V$ to the whole space $X$ so that they
belong to $\langle f\rangle$. Hence $f$ is not quasi-minimal.
\end{proof}

The familiar operations of addition and multiplication on the set
of real intervals is defined for $[\underline{a},\overline{a}],
[\underline{b},\overline{b}]\in\mathbb{I\,}{\mathbb{\overline{R}}}$
as follows:
\begin{eqnarray*}
&&\hspace{-4mm}[\underline{a},\overline{a}] +
[\underline{b},\overline{b}]\! =\! \{a+b:a\!\in\!
[\underline{a},\overline{a}], b\!\in\!
[\underline{b},\overline{b}]\}\!=\![\underline{a} +
\underline{b},\overline{a} + \overline{b}],
\\
&&\hspace{-4mm}[\underline{a}, \overline{a}]\!\times\!
[\underline{b},\overline{b}]\!=\! \{ab\!:\!a\!\in\!
[\underline{a},\overline{a}], b\!\in\!
[\underline{b},\overline{b}]\}\!=\!
[\min\{\underline{a}\underline{b},\underline{a}\overline{b},\overline{a}\underline{b},\overline{a}\overline{b}\},
\max\{\underline{a}\underline{b},\underline{a}\overline{b},\overline{a}\underline{b},\overline{a}\overline{b}\}],
\end{eqnarray*}
Ambiguities related to $\pm\infty$ are resolved in a way which
guarantees inclusion: \[ -\infty+(+\infty)=[-\infty,+\infty], \
0\times (+\infty)=[0,+\infty],\ 0\times(-\infty)=[-\infty,0]
\]

Point-wise operations for interval functions are defined in the
usual way:
\begin{equation}\label{f+g,fxg}
(f+g)(x)=f(x)+g(x)\ ,\ \ (f\times g)(x)=f(x)\times g(x)\ ,\ x\in
X.
\end{equation}

It is easy to see that the set of the S-continuous functions is
closed under the above point-wise operations while the set of
H-continuous functions is not. In earlier works by Markov, Sendov
and the author, see \cite{NMA2007}--\cite{Varna2005}, it was shown
that the algebraic operations on the set $\mathbb{H}_{ft}(X)$ of
all finite H-continuous function can be defined in such a way that
it is a linear space (the largest linear space of finite interval
functions) and a ring. As mentioned above, these results were
derived in the case when $X$ is a Baire space. Here we extend
these results in two ways:

(i) We assume that the domain $X$ is an arbitrary completely
regular topological space.

(ii) We consider the wider set $\mathbb{H}_{nf}$ of nearly finite
H-continuous functions.

These generalizations are motivated by the aim of the paper:
namely, constructing the rational extensions of $C(X)$ as rings of
functions defined on the same domain. More precisely, the problem
is considered in the same setting as in \cite{FGL}, that is, $X$ a
completely regular topological space. Furthermore, it is shown in
the sequel that the rational extensions of $C(X)$ and their metric
completions considered in \cite{FGL} cannot be all represented
within the realm of finite H-continuous function. Hence we need to
considered the larger set $\mathbb{H}_{nf}(X)$.

Let us recall that an H-continuous function $f$ is called {\it
nearly finite} if the set
\[
\Gamma_f=\{x\in X:-\infty\in f(x)\ or +\infty\in f(x)\}
\]
is closed and nowhere dense. The set $\mathbb{H}_{nf}(X)$ has
important applications in the Analysis of PDEs within the context
of the Order Completion Method, \cite{RosingerOber}. It turns out
that the solutions of large classes of systems of nonlinear PDEs
can be assimilated with nearly finite H-continuous functions, see
\cite{AngRos}, \cite{AngRos2}. The definition of the operations on
$\mathbb{H}_{nf}(X)$ are based on the following theorem.

\begin{theorem}\label{toperqcont}
For any $f,g\in\mathbb{H}_{nf}(X)$ the functions $f+g$ and
$f\times g$ are quasi-minimal S-continuous functions.
\end{theorem}

\begin{proof}
The proofs for $f+g$ and $f\times g$ use similar ideas based on
Theorem \ref{tqminchar}. We will present only the proof for
$f\times g$ which is slightly more technical. Assume the opposite.
In view of Theorem \ref{tqminchar}, this means that there exists
$\varepsilon>0$ and an open set $V$ such that $V\subset W_{f\times
g,\varepsilon}(f)$. Furthermore, since $f$ and $g$ are nearly
finite, the set $V\setminus(\Gamma(f)\cup\Gamma(g))$ is also open
and nonempty. Let $a\in V\setminus(\Gamma(f)\cup\Gamma(g))$. It is
easy to see that the functions $|f|$ defined by $|f|(x)=|f(x)|$,
$x\in X$, is an upper semi-continuous function. Therefore, there
exists an open set $D_1$ such that $a\in D_1\subset V
\setminus(\Gamma(f)\cup\Gamma(g))$ and $|f(x)|<|f(a)|+1$, $x\in
D_1$. Similarly, there exists an open set $D_2$ such that $a\in
D_2\subset V \setminus(\Gamma(f)\cup\Gamma(g))$ and
$|g(x)|<|g(a)|+1$, $x\in D_2$. Denote $D=D_1\cap D_2$. Then, using
a well known inequality about the width of a product of intervals,
see \cite{Alefeld}, for every $x\in D$ we obtain
\begin{eqnarray*}
\varepsilon&\leq& w((f\times g)(x))=w(f(x)\times g(x))\ \leq\
w(f(x))|g(x)|+w(g(x))|f(x)|\\
&\leq& w(f(x))(|g(a)|+1)+w(g(x))(|f(a)|+1)\ \leq \
(w(f(x))+w(g(x))m,
\end{eqnarray*}
where $m=\max\{|f(a)|+1~,~|g(a)|+1\}$. This implies
\[
D \subset W_{f,\frac{\varepsilon}{2m}}\cup
W_{g,\frac{\varepsilon}{2m}}.
\]
Therefore, at least one of the sets $W_{f,\frac{\varepsilon}{2m}}$
and $W_{g,\frac{\varepsilon}{2m}}$ has a nonempty open subset.
However, by Theorem \ref{tqminchar} this is impossible. The
obtained contradiction completes the proof.
\end{proof}

Theorem \ref{toperqcont} implies that for every
$f,g\in\mathbb{H}_{nf}(X)$ the sets $[f+g]$ and $[f\times g]$
contain one element each. Then we define addition and
multiplication on $\mathbb{H}_{nf}(X)$ by
\begin{eqnarray}
f\oplus g = [f+g]\label{defadd}\\
f\otimes g= [f\times g]\label{defmult}
\end{eqnarray}
Here and in the sequel we denote the operations by $\oplus$ and
$\otimes$ to distinguish them from the point-wise operations
denoted earlier by $+$ and $\times$. Equivalently to
(\ref{defadd})--(\ref{defmult}), one can say that $f\oplus g$ is
the unique H-continuous function contained in $f+g$ and that
$f\otimes g$ is the unique H-continuous function contained in
$f\times g$. It should be noted that the operations $\oplus$ and
$\otimes$ may coincide with $+$ and $\times$ respectively for some
values of the arguments. In particular,
\begin{equation}\label{opercoincide}
f\oplus g=f+g,\ f\otimes g=f\times g\ \mbox{ for }\ f\in C(X),\
g\in\mathbb{H}_{nf}(X).
\end{equation}

\begin{theorem}
The set $\mathbb{H}_{nf}(X)$ is a commutative ring with respect to
the operations $\oplus$ and $\otimes$.
\end{theorem}
The proof will be omitted. It involves standard techniques and is
partially discussed in \cite{RC} for the case of finite functions.
The zero and the identity in $\mathbb{H}_{nf}(X)$ are the constant
functions with values 0 and 1 respectively. We will denote them by
0 and 1 with the context showing whether we mean a constant
function or the respective real number. The multiplicative inverse
of $f\in \mathbb{H}_{nf}(X)$, whenever it exists, is denoted by
$\displaystyle\frac{1}{f}$. The non-zero-divisors in the ring
$\mathbb{H}_{nf}(X)$ can be characterized similarly to the ring
$C(X)$. However, unlike $C(X)$ all non-zero-divisors are
invertible. More precisely,
\begin{equation}\label{nonzerodivizor}
f \mbox{ is a non-zero-divisor}\Longleftrightarrow Z(f) \mbox{ is
nowhere dense in } X\Longleftrightarrow \exists g\in
\mathbb{H}_{nf}(X):f\otimes g=1
\end{equation}
where $Z(f)$ is the zero set of the function $f$ given by
$Z(f)=\{x\in X:0\in f(x)\}$. We denote the inverse of $f$, that
is, the function $g$ in (\ref{nonzerodivizor}) above, by
$\displaystyle\frac{1}{f}$. If
$f(x)=[\underline{f}(x),\overline{f}(x)$, $x\in X$, then we have
\begin{equation}\label{inverse}
\frac{1}{f}(x)=\left[\frac{1}{\overline{f}(x)}\frac{1}{\underline{f}(x)}\right]\
,\ \ x\in coz(f)=X\setminus Z(f).
\end{equation}
Note that in view of the property (\ref{ident}), the equality
(\ref{inverse}) determines $\displaystyle\frac{1}{f}$ in a unique
way because $coz(f)$ is dense in $X$.

Let $D$ be an open subset of $X$. The restriction $f|_D$ of $f$ on
$D$ is an H-continuous function on $D$, see \cite{Sozopol}. More
precisely,
\begin{equation}\label{restriction}
f\in\mathbb{H}_{nf}(X)\ \Longrightarrow \
f|_D\in\mathbb{H}_{nf}(D).
\end{equation}
Since $\mathbb{H}_{nf}(D)$ is also a ring it is useful to remark
that for any $f,g\in\mathbb{H}_{nf}(X)$ we have
\begin{equation}\label{restoperations}
(f\oplus g)|_D=f|_D\oplus g|_D\ , \ (f\otimes g)|_D=f|_D\otimes
g|_D.
\end{equation}
We will also use the following property, \cite{Sozopol}:
\begin{eqnarray}
&&\mbox{for any dense subset $D$ of $X$ and
$g\in\mathbb{H}_{nf}(D)$ there exists a unique}\nonumber\\
&&\mbox{function $f\in\mathbb{H}_{nf}(X)$ such that
$f|_D=g$.}\label{extension}
\end{eqnarray}

\section{Representing the Rational Extensions of $C(X)$}

The zero set $Z(f)$ and the cozero set $coz(f)$ of
$f\in\mathbb{H}_{nf}(X)$ generalize the respective concepts for
continuous function and play an important role in the ring
$\mathbb{H}_{nf}(X)$ as suggested by (\ref{nonzerodivizor}) and
(\ref{inverse}). This is further demonstrated in the following
useful lemma which extends the respective result in \cite[Section
2.2]{FGL}.

\begin{lemma}\label{theoIdealDense} Let $H_a(X)$ be a ring of
H-continuous functions such that $C(X)\subseteq H_a(X)\subseteq
\mathbb{H}_{nf}(X)$. An ideal $P$ of $H_a(X)$ is rationally dense
in $H_a(X)$ if and only if $coz(P)$ is a dense subset of $X$.
\end{lemma}

Any ideal $P$ of a ring $A$ is also an A-module. The rational
completeness of a ring can be characterized in terms of the
A-homomorphisms from the rationally dense ideals of $A$ to $A$ as
shown in the next theorem \cite{Lambek1966}.

\begin{theorem}\label{theoRComplCond}
A ring $A$ is rationally complete if for every rationally dense
ideal $P$ of $A$ and an $A$-homomorphism $\varphi:P\to A$ there
exists $s\in P$ such that $\varphi(p)=sp$, $p\in P$.
\end{theorem}

In the sequel we refer to the A-homomorphism shortly as
homomorphisms.

\begin{theorem}\label{theoHRcompl}
The ring $\mathbb{H}_{nf}(X)$ is rationally complete.
\end{theorem}
\begin{proof}
We use Theorem \ref{theoRComplCond}. Let $P$ be an ideal of
$\mathbb{H}_{nf}(X)$ and $\varphi:P\to \mathbb{H}_{nf}(X)$ a
homomorphism. Let $p\in P$. Consider the ring
$\mathbb{H}_{nf}(coz(p))$. By (\ref{inverse}), $p$ is an
invertible element of $\mathbb{H}_{nf}(coz(p))$. Since
$\phi(p)|_D\in\mathbb{H}_{nf}(coz(p))$, see (\ref{restriction}),
we can consider the function
$\psi_p=\frac{1}{p}\otimes\phi(p)|_D\in\mathbb{H}_{nf}(coz(p))$.

Now we define the function $\psi\in\mathbb{H}_{nf}(X)$ in the
following way. For any $x\in coz(P)$ select $p\in P$ such that
$0\notin p(x)$. Then
\[
\psi(x)=\psi_p(x)
\]
It is easy to see that the definition does not depend on the
function $p$. Indeed, let $q\in P$ be such that $0\notin q(x)$.
Since $\varphi$ is a homomorphism we have
\begin{equation}\label{theoHRcompl1}
\varphi(p)\otimes q=p\otimes\varphi(q).
\end{equation}
Denote $D=coz(p)\cap coz(q)$. Clearly $D$ is an open neighborhood
of $x$. Using (\ref{restoperations}) we have
\[
\varphi(p)|_D\otimes q|_D=p|_D\otimes\varphi(q)|_D,
\]
which implies
\[
\frac{1}{p}|_D\otimes\varphi(p)|_D=\frac{1}{q}|_D\otimes\varphi(q)|_D.
\]
Therefore $\phi_p(y)=\phi_q(y)$, $y\in D$. In particular,
$\phi_p(x)=\phi_q(x)$.

Now $\psi$ is defined on $coz(P)$ and it is easy to see that
$\psi\in\mathbb{H}_{nf}(coz(P))$. Since $coz(P)$ is dense in $X$,
see Lemma \ref{theoIdealDense}, using (\ref{extension}) the
function $\psi$ can be defined on the rest of the set $X$ in a
unique way so that $\psi\in\mathbb{H}_{nf}(X))$.

We will show that $\varphi(p)=\psi\otimes p$, $p\in P$. Let $p\in
P$. We have
\[
(\psi\otimes p)|_{coz(p)}=(\psi_p\otimes
p|_{coz(p)})=\left(\frac{1}{p}\right)|_{coz(p)}\otimes\varphi(p)|_{coz(p)}\otimes
p|_{coz(p)}=\varphi(p)|_{coz(p)}
\]
Then, using also (\ref{extension}), we obtain
\begin{equation}\label{theoHRcompl2}
(\psi\otimes p)(x)=\varphi(p)(x)\ ,\ \ x\in\overline{coz(p)},
\end{equation}
where $\overline{coz(p)}$ denotes the topological closure of the
set $coz(p)$. Applying standard techniques based on the minimality
property of H-continuous functions one can obtain that
$\varphi(p)(x)=0$ for $x\in X\setminus\overline{cos(p)}\subset
Z(p)$. Then we have
\begin{equation}\label{theoHRcompl3}
\varphi(p)(x)=0\in\psi(x)\times p(x)\ , \ \ x\in
X\setminus\overline{coz(p)}
\end{equation}
From (\ref{theoHRcompl2}) and (\ref{theoHRcompl3}) it follows
\begin{equation}\label{theoHRcompl4}
\varphi(p)(x)\subseteq\psi(x)\times p(x)\ ,\ \ x\in X.
\end{equation}
By the definition of the operation $\otimes$ the inclusion
(\ref{theoHRcompl4}) implies $\varphi(p)=\psi\otimes p$, which
completes the proof.
\end{proof}

The rational completeness of $\mathbb{H}_{nf}(X)$ implies that any
rational extension of any subring of $\mathbb{H}_{nf}(X))$ is a
subring of $\mathbb{H}_{nf}(X)$. In particular this applies to
$C(X)$, where the respective maximal ring of quotients and
classical ring of quotients are characterized in the next theorem.
As in the classical theory we call a subset $V$ of $X$ a zero set
if there exists $f\in C(X)$ such that $V=Z(f)$.

\begin{theorem}[Representation Theorem]\label{theoRatExt}
The ring of quotients and the classical ring of quotients of
$C(X)$ are the following subrings of $\mathbb{H}_{nf}(X))$:
\begin{eqnarray*}
&a)&\mathcal{Q}(C(X))=\mathbb{H}_{nd}(X)=\{f\in\mathbb{H}(X):\overline{W_f}\mbox{
is nowhere dense}\}\\
&b)&\mathcal{Q}_{cl}(C(X))=\mathbb{H}_{sz}(X)=\{f\in\mathbb{H}(X):W_f\mbox{
is a subset of a nowhere dense zero set}\}
\end{eqnarray*}
\end{theorem}
\begin{proof}
a) First we need to show that $\mathbb{H}_{nd}(X)$ is a ring of
quotients of $C(X)$. In terms of the definition we have to prove
that for any $\phi,\psi\in\mathbb{H}_{nd}(X))$, $\psi\neq 0$,
there exists $f\in C(X)$ such that $\phi\otimes f\in C(X)$ and
$\psi\otimes f\neq 0$. Since $\psi\neq 0$ the open set $coz(\psi)$
is not empty. Using that $W_\phi$ and $\Gamma_\phi$ are closed
nowhere dense sets we have $coz(\psi)\setminus(W_\phi\cup
\Gamma_\phi)\neq\emptyset$. Let $a\in
coz(\psi)\setminus(W_\phi\cup \Gamma_\phi)$. By the complete
regularity of $X$: (i) there exists a neighborhood $V$ of $x$ such
that $\overline{V}\subset coz(\psi)\setminus(W_\phi\cap
\Gamma_\phi)$; (ii) there exists a function $f\in C(X)$ such that
$f(a)=1$ and $f(x)=0$ for $x\in X\setminus V$. We have that
$\psi\times f$ does not have zeros in a neighborhood of $a$,
Therefore $\psi\otimes f\neq 0$. We prove next that $\phi\otimes
f\in C(X)$. Indeed, $\phi(x)\times f(x)=0$ for $x$ in the open set
$ (X\setminus\overline{V})\setminus (W_\phi\cup\Gamma_\phi)$ which
is dense in the open set $X\setminus\overline{V}$. Therefore
$(\phi\otimes f)|_{X\setminus\overline{V}}=0$, which implies
$(\phi\otimes f)|_{X\setminus\overline{V}}\in
C(X\setminus\overline{V})$. Obviously, $(\phi\otimes
f)|_{X\setminus(W_\phi\cup\Gamma_\phi)}\in
C(X\setminus(W_\phi\cup\Gamma_\phi)$. Hence, $\phi\otimes f\in
C(X)$. Therefore, $\mathbb{H}_{nd}(X)$ is a ring of quotients of
$C(X)$. It is the maximal ring of quotients of $C(X)$ if and only
if it is rationally complete. The proof of the rational
completeness of $\mathbb{H}_{nd}(X)$ is done in a similar way as
for $\mathbb{H}_{nf}(X)$ and will be omitted.

b) Let $f,g\in C(X)$, $Z(g)$ - nowhere dense in $X$. Then using
(\ref{nonzerodivizor}) we obtain
$\displaystyle\frac{f}{g}=f\otimes \frac{1}{g}\in
\mathbb{H}_{nf}(X)$. Moreover, by (\ref{inverse}) we have
$W_\frac{f}{g}\subseteq Z(g)$ which implies
$\displaystyle\frac{f}{g}\in\mathbb{H}_{nd}(X)$. Therefore
$\mathcal{Q}_{cl}(C(X))\subseteq\mathbb{H}_{nd}(X)$. Now we will
prove the inverse inclusion. Let $f\in\mathbb{H}_{nd}(X)$. Then
there exists $g\in C(X)$ such that $Z(g)$ is nowhere dense on $X$
and $W_f\subseteq Z(g)$. Consider the functions
\begin{eqnarray*}
\phi&=&\frac{f\otimes g}{1+f\otimes f}\\
\psi&=&\frac{g}{1+f\otimes f}.
\end{eqnarray*}
It easy to see that $|\phi|\leq|g|$ and $|\psi|\leq|g|$, which
implies that $\Gamma_\phi=\Gamma_\psi=\emptyset$. Furthermore,
since $\phi(x)=\psi(x)=0$ for all $x\in W_f$ we have
$W_\phi=W_\psi=\emptyset$. Hence, $\phi,\psi\in C(X)$. Since
$Z(\psi)=Z(g)$ is nowhere dense in $X$, the function $\psi$ is an
invertible element of $\mathbb{H}_{nd}(X)$. Then $\displaystyle
f=\frac{\phi}{\psi}\in\mathcal{Q}_{cl}(C(X))$, which completes the
proof.
\end{proof}

\section{Representing the metric completions of $\mathcal{Q}(C(X))$ and
$\mathcal{Q}_{cl}(C(X))$}

The metric $\rho$ on $C(X)$ given in (\ref{rho}) can be extended
to $\mathbb{H}_{nf}(X)$ as follows
\begin{equation}\label{rhoext}
\rho(f,g)=\sup_{x\in
X\setminus(\Gamma_f\cup\Gamma_g)}\frac{|f\ominus g|}{1+|f\ominus
g|},
\end{equation}
where $f\ominus g=f\oplus (-1)g$.

\begin{theorem}\label{theoMcomplete}
The set $\mathbb{H}_{nf}(X)$ is a complete metric space with
respect to the metric $\rho$ in (\ref{rhoext}).
\end{theorem}
\begin{proof}
Verifying that $\rho$ satisfies the axioms of a metric uses
standard techniques and will be omitted. We will prove the
completeness by using that $\mathbb{H}_{nf}(X)$ is a Dedekind
complete latticet with respect to the usual point-wise order, see
\cite{AngRos,AngRos2}. Furthermore, $\mathbb{H}_{nf}(X)$ is also a
vector lattice with respect to the addition $\oplus$ and the
multiplication by constants. This can be shown similarly to
\cite{Oconv}, where the case of finite H-continuous functions is
considered. The following implication for any $\varepsilon\in
(0,1)$ is easy to obtain and is useful in the proof
\begin{equation}\label{convorder}
\rho(f,g)<\varepsilon \Longleftrightarrow
\frac{-\varepsilon}{1-\varepsilon}\leq f\ominus g \leq
\frac{\varepsilon}{1-\varepsilon} \Longleftrightarrow
g\ominus\frac{\varepsilon}{1-\varepsilon}\leq f \leq
g\oplus\frac{\varepsilon}{1-\varepsilon}.
\end{equation}
Let $(f_\lambda)_{\lambda\in\Lambda}$ be a Cauchy net on
$\mathbb{H}_{nf}(X)$. There exists $\mu\in\Lambda$ such that
$\rho(f_\lambda,f_\mu)<0.5$. Then by (\ref{convorder}) the net
$(f_\lambda)_{\lambda\geq \mu}$ is bounded. Due to the Dedekind
order completeness of $\mathbb{H}_{nf}(X)$ the following infima
and suprema exist
\begin{eqnarray*}
\phi_\lambda&=&\inf\{f_\nu:\nu\geq\lambda\}\ , \ \ \lambda\geq\mu, \\
\psi_\lambda&=&\sup\{f_\nu:\nu\geq\lambda\}\ , \ \
\lambda\geq\mu,\\
\phi&=&\sup\{\phi_\lambda:\lambda\geq\mu\}\\
\psi&=&\inf\{\psi_\lambda:\lambda\geq\mu\}
\end{eqnarray*}
Let $\varepsilon\in(0,1)$. There exists $\lambda_\varepsilon$ such
that $\rho(f_\lambda,f_\nu)<\varepsilon$,
$\lambda,\nu\geq\lambda_\varepsilon$. It follows from
(\ref{convorder}) that
\[
f_{\lambda_\varepsilon}\ominus\frac{\varepsilon}{1-\varepsilon}\leq
f_\nu\leq
f_{\lambda_\varepsilon}\oplus\frac{\varepsilon}{1-\varepsilon},\
\nu\geq\lambda_\varepsilon.
\]
Therefore,
\[
f_{\lambda_\varepsilon}\ominus\frac{\varepsilon}{1-\varepsilon}\leq
\phi_\lambda\leq\psi_\lambda\leq
f_{\lambda_\varepsilon}\oplus\frac{\varepsilon}{1-\varepsilon}, \
\lambda\geq\lambda_\varepsilon,
\]
which implies
\begin{equation}\label{theoMcomplete1}
0\leq\psi_\lambda\ominus\phi_\lambda\leq
\frac{2\varepsilon}{1-\varepsilon},\
\lambda\geq\lambda_\varepsilon.
\end{equation}
Taking a supremum on $\lambda$ and considering that $\varepsilon$
is arbitrary we obtain $\phi=\psi$.

Further, from the inequalities
\begin{eqnarray*}
&&\phi_\lambda\leq f_\lambda\leq \psi_\lambda\\
&&\phi_\lambda\leq \phi\leq \psi_\lambda
\end{eqnarray*}
and (\ref{theoMcomplete1}) we obtain
\[
|f_\lambda\ominus\phi|\leq |\psi_\lambda\ominus\phi_\lambda|\leq
\frac{2\varepsilon}{1-\varepsilon},\
\lambda\geq\lambda_\varepsilon.
\]
or equivalently $\rho(f_\lambda,\phi)<\varepsilon$.  This implies
that $\lim\limits_\lambda f_\lambda=\phi$, which completes the
proof.
\end{proof}

Since the ring $\mathbb{H}_{nf}(X)$ is rationally complete, see
Theorem \ref{theoHRcompl}, as well as complete with respect to the
metric (\ref{rhoext}), see Theorem \ref{theoMcomplete}, it
contains all rings of quotients of $C(X)$ as well as their metric
completions. In particular, representation of the metric
completion $\overline{\mathcal{Q}(C(X))}$ of $\mathcal{Q}(C(X))$
is given in the next theorem.

\begin{theorem}\label{theoMcomplC}
The completion of the ring of quotients of $C(X)$ is
$\mathbb{H}_{nf}(X))$, that is,
\[
\overline{\mathcal{Q}(C(X))}=\mathbb{H}_{nf}(X)
\]
\end{theorem}
\begin{proof}
Since the completeness of $\mathbb{H}_{nf}(X))$ has already been
proved, we only need to show that $\mathcal{Q}(C(X))$ is dense in
$\mathbb{H}_{nf}(X)$. Using the representation of
$\mathcal{Q}(C(X))$ given in Theorem \ref{theoRatExt},
equivalently, we need to show that $\mathbb{H}_{nd}(X)$ is dense
in $\mathbb{H}_{nf}(X))$. Let
$f=[\underline{f},\overline{f}]\in\mathbb{H}_{nf}(X)$ and let
$n\in\mathbb{N}$. We have
\begin{equation}\label{theoMcomplC1}
\overline{f}(x)-\frac{1}{n}\leq\underline{f}(x)+\frac{1}{n}\ ,\ \
x\in X\setminus(W_{f,\frac{1}{n}}\cup\Gamma_f).
\end{equation}
Since the function on the left side of the inequality
(\ref{theoMcomplC1}) is upper semi-continuous while the function
on the right side is lower semi-continuous by the well known
Theorem of Han there exists $f_n\in
C(X\setminus(W_{f,\frac{1}{n}}\cap\Gamma_f))$ such that
\begin{equation}\label{theoMcomplC2}
\overline{f}(x)-\frac{1}{n}\leq
f_n(x)\leq\underline{f}(x)+\frac{1}{n}\ ,\ \ x\in
X\setminus(W_{f,\frac{1}{n}}\cup\Gamma_f).
\end{equation}
The set $X\setminus(W_{f,\frac{1}{n}}\cup\Gamma_f)$ is an open and
dense subset of $X$ because $W_{f,\frac{1}{n}}$ and $\Gamma_f$ are
closed nowhere dense sets. Hence $f_n$ can be extended in a unique
way to $X$ so that it is H-continuous on $X$. Since this extended
function may assume interval values or values involving
$\pm\infty$ only on the closed nowhere dense set
$W_{f,\frac{1}{n}}\cup\Gamma_f$ we have
$f_n\in\mathbb{H}_{nd}(X)$. Moreover, it follows from the
inequality (\ref{theoMcomplC2}) that
\[
\rho(f,f_n)\leq\sup_{x\in
X\setminus(W_{f,\frac{1}{n}}\cup\Gamma_f)}|f\ominus
f_n|\leq\frac{1}{n}.
\]
Hence, $\lim\limits_{n\to\infty}f_n=f$. Therefore
$\mathbb{H}_{nd}(X)$ is dense in $\mathbb{H}_{nf}(X)$.
\end{proof}

\section{Conclusion}

This paper gives an application of a class of interval functions,
namely, the Hausdorff continuous function, to the representation
of the rational extensions of $C(X)$ as well as their metric
completions. Traditionally, Interval Analysis is considered as
part of Numerical Analysis due to its major applications to
scientific computing. However, the study of the order, topological
and algebraic structure of the spaces of interval functions led to
some significant applications to other areas of mathematics, e.g.
Approximation Theory \cite{Sendov}, Analysis of PDEs
\cite{AngRos,AngRos2,JanHarm}, Real Analysis \cite{QM2004,Oconv}.
The results presented here are in the same line of applications.
It is shown that all rings of quotients of $C(X)$ and their matric
completions are subrings of the ring $\mathbb{H}_{nf}(X)$ of
nearly finite Hausdorff continuous functions. Thus,
$\mathbb{H}_{nf}(X))$ is the largest set of functions to which the
ring and metric structure of $C(X)$ can be extended in an
unambiguous way.


\begin{thebibliography}{9}

\bibitem{Alefeld} Alefeld G, Herzberger J, Introduction to Interval
Computations, Academic Press, NY,1983.

\bibitem{QM2004} Anguelov, R., Dedekind order
completion of C(X) by Hausdorff continuous functions. Quaestiones
Mathematicae {\bf 27}(2004) 153--170.

\bibitem{musco} Anguelov R., Kalenda O. F. K., The convergence space of minimal
usco mappings, Technical Report UPWT 2006/3, University of
Pretoria, 2006.

\bibitem{AnguelovMarkov1981} Anguelov, R., Markov, S.:
Extended Segment Analysis. Freiburger Intervall-Berichte {\bf 10}
(1981) 1--63.

\bibitem{NMA2007} Anguelov R., Markov S., Numerical Computations with Hausdorff
Continuous Functions,  In: T. Boyanov et al. (Eds.),  Numerical
Methods and Applications 2006 (NMA 2006), Lecture Notes in
Computer Science  4310,  Springer, 2007, 279-286.

\bibitem{Sozopol} Anguelov, R., Markov, S., Sendov, B.:
On the Normed Linear Space of Hausdorff Continuous Functions,
Lecture Notes in Computer Science, Vol. 3743, 2006, pp. 281--288.

\bibitem{RC}Anguelov, R., Markov, S., Sendov, B.: The Set of Hausdorff Continuous
Functions - the Largest Linear Space of Interval Functions,
Reliable Computing, Vol. 12, 2006, pp. 337--363.

\bibitem{Varna2005}Anguelov, R., Markov, S., Sendov, B.: Algebraic operations for
H-continuous functions, Proceedings of the International
Conference on Constructive Theory of Functions, Varna 2005, Marin
Drinov Acad. Publ. House, Sofia, 2006, pp. 35--44.

\bibitem{AnguelovRosinger2007}Anguelov, R., Rosinger, E. E., Solving Large Classes of
Nonlinear Systems of PDE's, Computers and Mathematics with
Applications, Vol. 53, 2007, pp. 491-507.

\bibitem{AngRos} Anguelov, R., Rosinger, E. E., Hausdorff Continuous Solutions of
Nonlinear PDEs through the Order Completion Method, Quaestiones
Mathematicae, Vol, 28(3), 2005, pp. 271--285.

\bibitem{AngRos2} Anguelov, R., Rosinger, E. E., Solving Large Classes of Nonlinear
Systems of PDE's, Computers and Mathematics with Applications,
Vol. 53, 2007, pp. 491--507.

\bibitem{Oconv} Anguelov, R., van der Walt, J. H.: Order Convergence Structure on
C(X), Quaestiones Mathematicae, Vol. 28 (4), pp.425--457.

\bibitem{Lambek1966}N.J. Fine, L. Gillman, J. Lambek, Rings of Quotients of Rings
of Functions, McGill University Press, Montreal, Quebec, 1966.

\bibitem{FGL} Fine, N. J., Gillman, L., Lambek, J., Rings of
Quotients of Rings of Functions, McGill University Press,
Montreal, 1965.

\bibitem{RosingerOber} M.B. Oberguggenberger, E.E. Rosinger, Solution
on Continuous Nonlinear PDEs through Order Completion,
North-Holland, Amsterdam, London, New York, Tokyo, 1994.

\bibitem{Sendov} Sendov, B., Hausdorff Approximations, Kluwer,
Boston, 1990.

\bibitem{JanHarm} van der Walt, J. H., The uniform order
convergence structure on $\mathcal{ML}(X)$, Quaestiones
Mathematicae, to appear.

\end{thebibliography}
\end{document}